\documentclass[final]{jotart}
\RequirePackage{amsmath}
\RequirePackage{amsopn}
\RequirePackage{amsfonts}
\RequirePackage{paralist}
\RequirePackage{amssymb}
\RequirePackage{amsthm}
\usepackage{graphicx}

\usepackage{amsrefs}
\renewcommand\MR[1]{\relax} 
\usepackage[all]{xy}
\usepackage{mathrsfs}

\theoremstyle{proclaim}
\newtheorem{thm}{Theorem}[section]
\numberwithin{equation}{section}

\newtheorem{cor}[thm]{Corollary}
\newtheorem{lemma}[thm]{Lemma}
\newtheorem{prop}[thm]{Proposition}

\theoremstyle{statement}

\newtheorem{remark}[thm]{Remark}
\newtheorem{example}[thm]{Example}
\newtheorem{mycomment}[thm]{Comment}
{\end{mycomment}\endgroup}
\def\mathcs{C^{*}}
\newcommand{\cs}{\ensuremath{\mathcs}}

\DeclareMathSymbol{\rtimes}{\mathbin}{AMSb}{"6F}
\newcommand{\ib}{im\-prim\-i\-tiv\-ity bi\-mod\-u\-le}

\newcommand{\sme}{\,\mathord{\mathop{\text{--}}\nolimits_{\relax}}\,}

\def\C{\mathbf{C}}
\def\T{\mathbf{T}}
\def\Z{\mathbf{Z}}
\def\K{\mathcal{K}}

\DeclareMathOperator{\Ad}{Ad}

\DeclareMathOperator{\Aut}{Aut}

\DeclareMathOperator{\id}{id}
\def\set#1{\{\,#1\,\}}
\newcommand\sset[1]{\{#1\}}
\let\tensor=\otimes

\makeatletter
\def\labelenumi{\textnormal{(\@alph\c@enumi)}}
\def\theenumi{\@alph \c@enumi}
\def\labelenumii{\textnormal{(\@roman\c@enumii)}}
\def\theenumii{\@roman \c@enumii}
\newcount\charno
\def\alphapart#1{\charno=96
\advance\charno by#1\char\charno}

\makeatother

\def\<{\langle}
\def\>{\rangle}
\let\ipscriptstyle=\scriptscriptstyle
\def\lipsqueeze{{\mskip -3.0mu}}
\def\ripsqueeze{{\mskip -3.0mu}}
\def\ipcomma{\nobreak\mathrel{,}\nobreak}
\newbox\ipstrutbox
\setbox\ipstrutbox=\hbox{\vrule height8.5pt
width 0pt}
\def\ipstrut{\copy\ipstrutbox}
\def\lip#1<#2,#3>{\mathopen{\relax_{\ipstrut\ipscriptstyle{
#1}}\lipsqueeze
\langle} #2\ipcomma #3 \rangle}
\def\blip#1<#2,#3>{\mathopen{\relax_{\ipstrut
\ipscriptstyle{ #1}}\lipsqueeze\bigl\langle} #2\ipcomma #3 \bigr\rangle}
\def\rip#1<#2,#3>{\langle #2\ipcomma #3
\rangle_{\ripsqueeze\ipstrut\ipscriptstyle{#1}}}
\def\brip#1<#2,#3>{\bigl\langle #2\ipcomma #3
\bigr\rangle_{\ripsqueeze\ipstrut\ipscriptstyle{#1}}}
\def\angsqueeze{\mskip -6mu}
\def\smangsqueeze{\mskip -3.7mu}
\def\trip#1<#2,#3>{\langle\smangsqueeze\langle #2\ipcomma #3
\rangle\smangsqueeze\rangle_{\ripsqueeze\ipstrut\ipscriptstyle{#1}}}
\def\btrip#1<#2,#3>{\bigl\langle\angsqueeze\bigl\langle #2\ipcomma
#3
\bigr\rangle
\angsqueeze\bigr\rangle_{\ripsqueeze\ipstrut\ipscriptstyle{#1}}}
\def\tlip#1<#2,#3>{\mathopen{\relax_{\ipstrut\ipscriptstyle{
#1}}\lipsqueeze \langle\smangsqueeze\langle} #2\ipcomma #3
\rangle\smangsqueeze\rangle}
\def\btlip#1<#2,#3>{\mathopen{\relax_{\ipstrut\ipscriptstyle{
#1}}\lipsqueeze
\bigl\langle\angsqueeze\bigl\langle} #2\ipcomma #3
\bigr\rangle\angsqueeze\bigr\rangle}

\def\ip(#1|#2){(#1\mid #2)}
\def\bip(#1|#2){\bigl(#1 \mid #2\bigr)}
\def\Bip(#1|#2){\Bigl( #1 \bigm| #2 \Bigr)}
\expandafter\ifx\csname BibSpec\endcsname\relax\else
\BibSpec{collection.article}{%
    +{}  {\PrintAuthors}                {author}
    +{,} { \textit}                     {title}
    +{.} { }                            {part}
    +{:} { \textit}                     {subtitle}
    +{,} { \PrintContributions}         {contribution}
    +{,} { \PrintConference}            {conference}
    +{}  {\PrintBook}                   {book}
    +{,} { }                            {booktitle}
    +{,} { }                            {series}
    +{,} { \voltext}                    {volume}
    +{,} { }                            {publisher}
    +{,} { }                            {organization}
    +{,} { }                            {address}
    +{,} { \PrintDateB}                 {date}
    +{,} { pp.~}                        {pages}
    +{,} { }                            {status}
    +{,} { \PrintDOI}                   {doi}
    +{,} { available at \eprint}        {eprint}
    +{}  { \parenthesize}               {language}
    +{}  { \PrintTranslation}           {translation}
    +{;} { \PrintReprint}               {reprint}
    +{.} { }                            {note}
    +{.} {}                             {transition}
}
\BibSpec{article}{%
    +{}  {\PrintAuthors}                {author}
    +{,} { \textit}                     {title}
    +{.} { }                            {part}
    +{:} { \textit}                     {subtitle}
    +{,} { \PrintContributions}         {contribution}
    +{.} { \PrintPartials}              {partial}
    +{,} { }                            {journal}
    +{}  { \textbf}                     {volume}
    +{}  { \PrintDatePV}                {date}
    +{,} { \eprintpages}                {pages}
    +{,} { }                            {status}
    +{,} { \PrintDOI}                   {doi}
    +{,} { available at \eprint}        {eprint}
    +{}  { \parenthesize}               {language}
    +{}  { \PrintTranslation}           {translation}
    +{;} { \PrintReprint}               {reprint}
    +{.} { }                            {note}
    +{.} {}                             {transition}
}
\BibSpec{book}{%
    +{}  {\PrintPrimary}                {transition}
    +{,} { \textit}                     {title}
    +{.} { }                            {part}
    +{:} { \textit}                     {subtitle}
    +{,} { \PrintEdition}               {edition}
    +{}  { \PrintEditorsB}              {editor}
    +{,} { \PrintTranslatorsC}          {translator}
    +{,} { \PrintContributions}         {contribution}
    +{,} { }                            {series}
    +{,} { \voltext}                    {volume}
    +{,} { }                            {publisher}
    +{,} { }                            {organization}
    +{,} { }                            {address}
    +{,} { pp.~}                        {pages}
    +{,} { \PrintDateB}                 {date}
    +{,} { }                            {status}
    +{}  { \parenthesize}               {language}
    +{}  { \PrintTranslation}           {translation}
    +{;} { \PrintReprint}               {reprint}
    +{.} { }                            {note}
    +{.} {}                             {transition}
}
\fi
\newcommand\go{G^{(0)}} 
\def\g[#1,#2]{{}_{G}[#1,#2]} 
\def\h[#1,#2]{[#1,#2]_{H}}

\newcommand\B{\mathscr{B}}
\newcommand\A{\mathscr{A}}
\def\sa_#1(#2,#3){\Gamma_{#1}(#2;#3)}
\newcommand\X{\mathsf{H}}
\newcommand\XX{\mathscr{H}}
\newcommand\KK{\mathscr{K}}
\newcommand\E{\mathscr{E}}
\newcommand\CC{\mathscr{C}}

\newcommand\Iso{\operatorname{Iso}}
\renewcommand\H{\mathcal{H}}
\newcommand\U{\mathscr{U}}
\newcommand\uG{\underline{G}}
\newcommand\uGo{\uG^{(0)}}
\newcommand\uA{\underline{\A}}

\newcommand\uj{\underline{j}}
\newcommand\alphau{\underline{\alpha}}
\newcommand\uE{\underline{E}}

\newcommand\tE{\uE}
\newcommand\iotau{\underline{\iota}}
\newcommand\pr{\operatorname{pr}}
\newcommand\eop{E^{\text{o}}}
\newcommand\basis[1]{\mathcal{#1}}
\newcommand\W{\basis{W}}
\newcommand\ue{\underline{e}}
\newcommand\ccge{C_{c}(G;E)}
\newcount\hours
\newcount\minutes       
\def\timeofday{
\hours=\time
\minutes=\hours
\divide\hours by60
\multiply\hours by60
\advance\minutes by-\hours
\divide\hours by60
\ifnum\hours>9\else0\fi\the\hours:\ifnum\minutes>9\else
0\fi\the\minutes}
\def\predate{\date{\the\day\ \ifcase\month\or
  January\or February\or March\or April\or May\or June\or July\or
        August\or September\or October\or November\or
           December\fi\ \the\year\ --- \timeofday\ --- Preliminary
                  Version}}
     
\usepackage[normalem]{ulem} 
\usepackage{color}
\definecolor{Dgreen}{cmyk}{0.93,0.33,0.92,0.25} 

\begin{document}

\title[Groupoid Crossed Products]{Groupoid 
Crossed Products of Continuous-Trace \cs-Algebras}
\author{Erik van Erp}
\address{Department of Mathematics \\ Dartmouth College \\ Hanover, NH
03755\\ USA}
\email{erikvanerp@dartmouth.edu}
\author{Dana P. Williams} 
\address{Department of Mathematics \\ Dartmouth College \\ Hanover, NH
03755 \\ USA}
\email{dana.williams@dartmouth.edu}


\date{3 September 2013}

\begin{abstract}
We show that if $(A,G,\alpha)$ is a groupoid dynamical system with $A$
continuous trace, then the crossed product $A\rtimes_{\alpha}G$ is
Morita equivalent to the \cs-algebra $\cs(\uG,\uE)$ of a twist $\uE$
over a groupoid $\uG$ equivalent to $G$.  This is a groupoid
analogue of the well known result for the crossed product of a group
acting on an elementary \cs-algebra.
\end{abstract}

\begin{subjclass}
  46L05, 46L55
\end{subjclass}

\begin{keywords}
  Groupoid \cs-algebras, continuous-trace \cs-algebras, crossed products
\end{keywords}
\maketitle

\section{Introduction}
\label{sec:introduction}

One of the basic results in the theory of crossed products of
\cs-algebras by groups is the result, due to Green
\cite{gre:am78}*{Theorem~18}, computing the crossed product
$A\rtimes_{\alpha}G$ when $A$ is elementary.  The primary object of
this note is to prove an analogue, up to Morita equivalence, of
Green's result for Groupoid crossed products.

For motivation, we recall some of the details of Green's result.  If
$(A,G,\alpha)$ is a dynamical system with $A=\K(\H)$ for a complex
Hilbert space $\H$, then there is a short exact sequence of locally
compact groups
\begin{equation}
  \label{eq:1}
  \xymatrix{1\ar[r]&\T\ar[r]^{i}&E\ar[r]^{j}&G\ar[r]&1}
\end{equation}
that arises as follows.  The unitary group $U(\H)$ acts on $A=\K(\H)$
by automorphisms, via $U\mapsto \Ad U$, and by Wigner's theorem every
automorphism of $A$ arises in this way.  The kernel of $U(\H)\to
\Aut\,A$ is the center $\T\cdot I_\H\cong \T$ of $U(\H)$.  The
algebraic isomorphism $U(\H)/\T\cong \Aut\,A$ is a homeomorhism if
$U(\H)$ is given the strong operator topology.  The action
$\alpha\colon G\to \mathrm{Aut}\,A$ gives rise to the sequence
\eqref{eq:1} via pullback,
\begin{equation*}
  \xymatrix{
    1\ar[r]&\T\ar[r]^{i}\ar[d]&E\ar[r]^{j}\ar[d]&G\ar[r]\ar[d]^\alpha&1\\
    1\ar[r]&\T\ar[r]&U(\H)\ar[r]^-{\Ad}&\mathrm{Aut}\,A\ar[r]&1.
  }
\end{equation*}
In other words, $E$ is the fibered product
\[ E:=\set{(s,U)\in G\times U(\H):\alpha_{s}=\Ad U},\] and $E$ is a
locally compact group if given the relative topology in $G\times
U(\H)$. (This is not trivial, because $U(\H)$ is not locally compact.
The construction is described in detail in \cite{wil:crossed}*{\S7.3
  \& \S D.3}.)

Since $i(\T)$ is central in $E$, every irreducible unitary
representation $\pi$ of $E$ has a single $\T$-type; i.e., for every
$\pi$ there is an integer $k\in \Z$ such that $\pi(i(z))=z^k
I_{\H_\pi}$ for $z\in \T$.  The unitary dual $\hat{E}$ decomposes as a
disjoint union of closed subsets according to $\T$-type.  The twisted
group $\cs$-algebra $\cs(G;E)$ is the quotient of $\cs(E)$
corresponding to $\T$-type $\pi(i(z))=\bar z I_{\H_{\pi}}$.

Green's result says that the crossed product $A\rtimes_{\alpha}G$ is
isomorphic to the tensor product $\cs(G;E)\tensor A$.  Since $A$ is
elementary, $A\rtimes_{\alpha}G$ is Morita equivalent to $\cs(G;E)$.

We want to exhibit the analogous result for groupoid crossed products
$\A\rtimes_{\alpha}G$ where $G$ is a second countable locally compact
Hausdorff groupoid and $\A$ is an upper semicontinuous \cs-bundle of
elementary \cs-algebras over the unit space $\go$.  Our techniques
require that the section algebra $A=\sa_{0}(\go,\A)$ be a separable
continuous-trace \cs-algebra.

The groupoid analogue of central extensions of $G$ are called either
\emph{$\T$-group\-oids} over $G$ or \emph{twists} over $G$.  A twist
over $G$ is a principal $\T$-bundle $j:E\to G$ where $E$ has a
groupoid structure making $j$ a groupoid homomorphism.  The associated
\cs-algebras $\cs(G;E)$ have been extensively studied
\cites{muhwil:jams04,muhwil:plms395,muhwil:ms92,kum:cjm86}.  In the
case where $A=\sa_{0}(\go,\A)$ has continuous trace with trivial
Dixmier-Douady class $\delta(A)$, our main result
(Theorem~\ref{thm-main}) says that there is a twist $E$ over $G$,
analogous to \eqref{eq:1}, such that $\A\rtimes_{\alpha}G$ is Morita
equivalent to $\cs(G;E)$.  If $\delta(A)$ is nontrivial, we must
replace $G$ by an equivalent groupoid $\uG$.  Then we show that
$\A\rtimes_{\alpha}G$ is Morita equivalent to $\cs(\uG;\uE)$ for an
appropriate twist $\uE$.

The converse also holds, and is much easier.  In
\S\ref{sec:partial-converse} we show that if $E$ is a twist over
$G$, there is a Hilbert $C_{0}(\go)$-module $\X$ and an action
$\alpha$ of $G$ on the generalized compacts $\K(\X)$ such that
$\cs(G;E)$ is Morita equivalent to $\K(\X)\rtimes_{\alpha}G$.

{\emergencystretch=10pt Our result is closely related to the work on the Brauer group in
\cite{kmrw:ajm98}.  The Brauer group $\mathrm{Br}(G)$ consists of
equivalence classes (appropriately defined) of $\cs$-dynamical systems
$(\A,G,\alpha)$ with a continuous trace \cs-algebra
$A=\sa_{0}(\go,\A)$.  The main result of \cite{kmrw:ajm98} is that
$\mathrm{Br}(G)$ is isomorphic to a group $\mathrm{Ext}(G,\T)$ whose
elements are pairs $(\uG,\uE)$ consisting of a groupoid $\uG$ that is
equivalent to $G$ and a twist $\uE$ of $\uG$, where the pairs
$(\uG,\uE)$ are subject to a subtle equivalence relation.\par}

The focus in \cite{kmrw:ajm98} is on establishing a group
isomorphism betweem $\mathrm{Br}(G)$ and $\mathrm{Ext}(G,\T)$, and the
difficulty resides in the precise equivalence relations that define
the two groups.  However, the relation between the (maximal or
reduced) \cs-algebras $\A\rtimes_\alpha G$ and $\cs(\uG;\uE)$ is not
considered there.

Our main tool here is the Equivalence Theorem for Fell bundles from
\cite{muhwil:dm08}. This has the advantage that explicit pre-\ib s for
our Morita equivalences can be read off from the formulas in
\cite{muhwil:dm08}.  More significantly, it follows from
\cite{simwil:nyjm13}*{Theorem~14} that our results pass to the reduced
algebras; that is, we also have a Morita equivalence of
$\A\rtimes_{\alpha,r} G$ and $\cs_r(\uG,\uE)$.  However, because we
use the Equivalence Theorem, all our results require separability,
which is not the case for Green's result.

In a different context, an equivalence class of groupoids represents a
{\em stack}, and a twist over a groupoid represents an $S^1$-{\em
  gerbe}.  The twisted $K$-theory of a stack twisted by a gerbe is, by
definition, the $K$-theory of the reduced $\cs$-algebra $\cs_r(G,E)$,
where $G$ is a groupoid representing the stack, and $E$ is a twist of
$G$ corresponding to the gerbe.  In \cite{txlg:acens12} Tu, Xu and
Laurent-Gengoux consider the twisted $K$-theory of {\em
  differentiable} stacks, i.e., the case where $G$ is a Lie groupoid
and $E$ is a smooth twist of $G$.  In the language of
\cite{txlg:acens12}, the present paper deals with (second countable)
{\em locally compact} stacks, and our results imply that, in that more
general context, the $K$-theory of a crossed product $\A\rtimes_\alpha
G$ is naturally isomorphic to the twisted $K$-theory of the stack
represented by $G$ for an appropriate gerbe, and vice versa.

\section{Preliminaries}
\label{sec:preliminaries}

For further references and results on upper semicontinuous \cs-algebra
bundles and upper semicontinuous Banach bundles we refer to
\cite{wil:crossed}*{Appendix C} and \cite{muhwil:dm08}*{Appendix A},
respectively; for groupoid crossed products, we refer to
\cite{muhwil:nyjm08}; and for Fell bundles and their associated
\cs-algebras, to \cite{muhwil:dm08}.

If $p:\B\to X$ is a Banach bundle, we write $B(x)$ for the Banach
space that is the fibre of $\B$ over $x$.  We write $\sa_{}(X,\B)$ for
the continuous sections of $\B$, and $\sa_{0}(X,\B)$ and
$\sa_{c}(X,\B)$ for the continuous sections vanishing at infinity or
with compact support, respectively.

We review the basic definitions for convenience.  In all that follows, $G$
will be a second countable locally compact Hausdorff groupoid with a
Haar system $\sset{\lambda^{u}}_{u\in\go}$.

\subsection{Groupoid Crossed Products}
\label{sec:group-cross-prod}
A groupoid dynamical system $(\A,G,\alpha)$ consists of an upper
semicontinuous \cs-bundle $p\colon \A\to \go$ with a continuous left
$G$-action
\[ \alpha\colon G\times_{s,p} \A:=\{(x,a)\in G\times \A\mid
s(x)=p(a)\} \to \A \] such that for every $x\in G$ the map
$\alpha_x(a):=\alpha(x,a)$ is an isomorphism of \cs-algebras
\[ \alpha_{x}:A(s(x))\to A(r(x)).\]
Then $\sa_{c}(G, r^{*}\A)$ is a $*$-algebra with respect to
\begin{equation*}
  f*g(x):= \int_{G}
  f(y)\alpha_{y}\bigl(g(y^{-1}x)\bigr)
  \,\lambda^{r(x)}(y) \quad\text{and}\quad
  f^{*}(x)=\alpha_{x}\bigl(f(x^{-1})^{*}\bigr) .
\end{equation*}
The crossed product $\A\rtimes_{\alpha}G$ is the completion of
$\sa_{c}(G,r^{*}\A)$ with respect to all suitably bounded
representations, and the reduced crossed product
$\A\rtimes_{\alpha,r}G$ is the completion of $\sa_{c}(G,r^{*}\A)$ with
respect to the regular representations.  (See \cite{muhwil:nyjm08}.)

\begin{remark}[Notation for Crossed Products]
  \label{rem-notation}
  A \cs-algebra $A$ can be given the structure of a $C_{0}(X)$-algebra
  if and only if there is an upper semicontinuous \cs-bundle $\A$ so
  that $A$ is $C_{0}(X)$-isomorphic to $\sa_{0}(X,\A)$
  \cite{wil:crossed}*{Theorem~C.26}.  If $(\A,\alpha,G)$ is a groupoid
  dynamical system and $A=\sa_{0}(\go,\A)$, then both
  $\A\rtimes_{\alpha}G$ and $A\rtimes_{\alpha}G$ are used to denote
  the crossed product.  We usually prefer the bundle notation
  $\A\rtimes_{\alpha}G$.
\end{remark}

\subsection{Twists}
\label{sec:twists}
A twist $E$ over $G$, or alternatively, a $\T$-groupoid over $G$, is a
central groupoid extension
\begin{equation}
  \label{eq:2}
  \xymatrix{\go\times\T\ar[r]^-{\iota}&E\ar[r]^{j}&G.}
\end{equation}
A central extension is one such that $\iota(r(e),z)e=e\iota(s(e),z)$
for all $e\in E$ and $z\in\T$.  In particular, $E$ admits a (left or
right) $\T$-action $z\cdot e :=\iota(r(e),z)e$.  Since \eqref{eq:2} is
meant to be an extension of topologicial groupoids, we are insisting
that $i$ is a homeomorphism onto the kernel of $j$, and that $j$ is
open and continuous.  In particular, $E$ is also a principal
$\T$-bundle over $G$.  Note that if $G$ is a group, then \eqref{eq:2}
is just a central extension of locally compact groups just as in
\eqref{eq:1}.

As in \cites{muhwil:plms395,muhwil:ms92}, we associate a \cs-algebra
to a twist \eqref{eq:2} as follows.  We let
\begin{equation}
  \label{eq:3}
  C_{c}(G;E)=\set{f\in C_{c}(E):\text{$f(ze)=zf(e)$ for all $z\in\T$ and
      $e\in E$}}.
\end{equation}
Then $C_{c}(G;E)$ becomes a $*$-algebra with respect to
\begin{equation*}
  f*g(e') =\int_{G} f(e)g(e^{-1}e')
  \,d\lambda^{r(e')}(j(e))\quad\text{and}\quad f^{*}(e)=\overline{f(e^{-1})}.
\end{equation*}
The integral that defines $f\ast g$ makes sense because for fixed
$e'\in E$ the expression $f(e)g(e^{-1}e')$ is a function of $j(e)\in
G$.  Its universal \cs-completion is denoted by $\cs(G;E)$ and its
completion with respect to its regular representations is denoted by
$\cs_{r}(G;E)$.

\begin{example}[Projective Representations]
  \label{ex-group}
  Suppose that $E$ is a twist over a \emph{group} $G$.  Then we get a
  Haar measure on $E=\T\times G$ 
  as the product of Haar measures on $\T$ and $G$ (normalized such
  that $\T$ has measure 1).  Then multiplication $f\ast g$ in
  $C_{c}(G;E)$ can be written as an integral over $E$,
  \begin{equation*}
    f*g(e') =\int_{E} f(e)g(e^{-1}e')
    \,d\lambda(e)
  \end{equation*}
  In other words, $C_{c}(G;E)$ is a sub $\ast$-algebra of the
  convolution algebra $C_c(E)$.  If $\pi$ is a unitary representation
  of $E$, then for $f\in C_c(G;E)$ and $z\in\T$ we find:
  \begin{align*}
    \pi(f) &= \int_E f(e)\pi(e)d\lambda(e) = \int_E zf(\bar z
  e)\pi(e)d\lambda(e) = z \int_E f(e)\pi(z e)d\lambda(e) \\
&=z\pi(z)\pi(f).
  \end{align*}
 In other words, if we decompose the Hilbert space
  $\H_\pi$ according to $\T$-type, then the restriction of $\pi$ to
  $C_c(G;E)$ is zero on all subspaces, excecpt the one with $\T$-type
  $\pi(z)=\bar z I_{\H_\pi}$.  Thus, if $G$ is a group then $\cs(G;E)$
  is the quotient of $\cs(E)$ corresponding to those unitary
  representations $\pi$ of $E$ that satisfy $\pi(i(z))=\bar z
  I_{\H_{\pi}}$.

  Central extensions of $G$ by $\T$ are classified (up to isomorphism)
  by $H^2(G,\T)$.  If $c$ is a Borel cross section for $j:E\to G$ such
  that $c(e)$ is the identity element of $E$, then the corresponding
  Borel $2$-cocycle $\omega\in Z^{2}(G,\T)$ is determined by
  $c(s)c(r)= \omega(s,r)c(sr)$.
  
  Recall that an {\em $\omega$-multiplier representation} of $G$ is a
  Borel map $\bar \pi:G\to U(\H)$ such that $\bar \pi(s)\bar\pi(r)
  =\omega(s,r)\bar\pi(sr)$.  Note that $\omega$-multiplier
  representations of $G$ and are in one-to-one correspondence with
  unitary representations $\pi$ of $E$ that satisfy $\pi(i(z))=z
  I_{\H}$.

  Therefore it is not surprising that $\cs(G;E)$ is isomorphic to
  $\cs(G,\bar \omega)$ where $\cs(G,\bar \omega)$ is the universal
  \cs-algebra for $\bar \omega$-multiplier representations of $G$ and
  $\bar \omega$ is the complex conjugate of $\omega$.

\end{example}

\subsection{\boldmath Fell Bundles and their \cs-Algebras}
\label{sec:fell-bundles-their}

{\emergencystretch=20pt A Fell bundle over a locally compact Hausdorff groupoid $G$ is an
upper semicontinuous Banach bundle $p:\B\to G$ equipped with a
continuous, bilinear, associative multiplication $(a,b)\mapsto ab$
from $\B^{(2)}=\set{(a,b)\in\B\times\B:(p(a),p(b))\in G^{(2)}}$ to
$\B$ such that the diagram
\[\xymatrix{\B^{(2)}\ar[r]\ar[d]_p&\B\ar[d]^p\\
  G^{(2)}\ar[r]&G }\]
commutes, and such that there is a continuous involution $b\mapsto
b^{*}$ from $\B$ to $\B$ such that
\[\xymatrix@C+2pc{\B\ar[r]^{b\mapsto b^*}\ar[d]_p&\B\ar[d]^p\\
  G\ar[r]_{x\mapsto x^{-1}}&G }\] commutes, and such that, as usual,
\[ (ab)^{*}=b^{*}a^{*}.\] These axioms imply that for a unit $u\in
\go$ the fiber $B(u)$ is a Banach $\ast$-algebra with respect to the
inherited operations, while an arbitrary fiber $B(x)$ is a left
$B(r(x))$ and right $B(s(x))$ bimodule.  Finally, for $\B$ to be a
Fell bundle it is required that the $\ast$-algebra $B(u)$ is a
\cs-algebra, while $B(x)$ must be a $B(r(x))\sme B(s(x))$-\ib\ when
given inner products
\begin{equation*}
  \lip B(r(x))<a,b>:=ab^{*}\quad\text{and}\quad \rip B(s(x))
  <a,b> :=a^{*}b.
\end{equation*}
If $\B$ is a Fell bundle over $G$ we can make $\sa_{c}(G,\B)$ into a
$*$-algebra in a straightforward way (provided $G$ has a Haar system).
That is, we define
\begin{equation*}
  f*g(x):=\int_{G}f(\eta)g(y^{-1}x)\,d\lambda^{r(x)}(y)
  \quad\text{and}\quad f^{*}(x):=f(x^{-1})^{*}.
\end{equation*}
We can then form the universal completion $\cs(G,\B)$ as well as the
reduced one $\cs_{r}(G,\B)$.

Fell bundles and their associated \cs-completions include virtually
all known \cs-algebras associated to dynamical systems (see
\cite{muhwil:dm08}*{\S2}).  We include the examples that are relevant
to our discussion here below.\par}

\begin{example}[Twists Revisited]
  \label{ex-twist-fell}
  Let $E$ be a twist over $G$ as in \S\ref{sec:twists}.  Then $E$ is a
  principal $\T$-bundle.  If we let $\B$ be the associated complex
  line bundle; that is, let $\B$ be the quotient of $E\times \C$ by
  the diagonal $\T$-action $z\cdot (e,\lambda)=(ze,\bar z\lambda)$.
  Then $\B$ is a line bundle over $G$ which we can treat as Fell
  bundle (see \cite{muhwil:dm08}*{Example~2.4}).  Note that the
  sections of $\B$ correspond to continuous functions on $E$ which
  satisfy
  \begin{equation*}
    f(ze)=\bar z f(e).
  \end{equation*}
  Comparing the above with \eqref{eq:3}, it is not hard to see that
  $\cs(\B)$ is isomorphic to $\cs(G;\eop)$ where $\eop$ is the
  conjugate $\T$-bundle to $E$.\footnote{This subtlety was overlooked
    in \cite{muhwil:dm08}*{Example~2.9} were it is erroneously claimed
    to be isomorphic to $\cs(G;E)$. As is often the case with
    mathematical constructs where a choice of sign or conjugate is
    involved, different authors make different choices.  We have
    chosen to keep our notation consistent with the published
    literature and not redefine $\cs(G;E)$ to suit the present
    circumstances.}
  
  To recover $\cs(G;E)$, we work with the line bundle $\CC$ associated
  to $\eop$.  Note that $\CC$ can be thought of as the quotient of
  $E\times \C$ with respect to the $\T$-action $z\cdot (e,\lambda)=(z
  e, z\lambda)$.  Then $\cs(G,\CC)\cong \cs(G;E)$, and it is not hard
  to see that $\cs_{r}(G,\CC)\cong \cs_{r}(G;E)$.
\end{example}

\begin{example}[Groupoid Crossed Products]
  \label{ex-cross-prod}
  Let $(\A,G,\alpha)$ be a groupoid dynamical system.  Then we can
  make $\B:=s^{*}\A=\set{(\gamma,a):a\in A(s(\gamma))}$ into a Fell
  bundle where $(\gamma,a)(\eta,b)=
  (\gamma\eta,\alpha_{\eta}^{-1}(a)b)$ and
  $(\gamma,a)^{*}=(\gamma^{-1},\alpha_{\gamma}^{-1}(a))$. Since the
  map $(\gamma,a)\mapsto (\alpha_{\gamma}(a),\gamma)$ is a Fell bundle
  isomorphism of $\B$ onto the bundle constructed in
  \cite{muhwil:dm08}*{Example~2.1}, it follows as in
  \cite{muhwil:dm08}*{Example~2.8}, that $\cs(G,\B)\cong
  \A\rtimes_{\alpha}G$. (We have used $\B=s^{*}\A$, rather than
  $r^{*}\A$ as in \cite{muhwil:dm08} as it makes some of the formulas
  in \S\ref{sec:main-theorem} a bit tidier.  This is also Muhly's
  original formulation from \cite{muh:cm01}*{\S3}.)  It follows as in
  \cite{simwil:nyjm13}*{Example~11} that $\cs_{r}(G,\B)\cong
  \A\rtimes_{\alpha,r}G$.
\end{example}

\section{Building the $\T$-Groupoid}
\label{sec:phill-raeb-appr}

In the next section we prove that, for continuous trace $\A$ and
locally compact second countable $G$, the crossed product
$\A\rtimes_{\alpha}G$ is Morita equivalent to $\cs(\uG;\uE)$.  In this
section we construct the pair $(\uG,\uE)$.  The construction can be
summarized as follows.

If the Dixmier-Douady invariant of $\A$ is zero, then $A\cong \K(\X)$
for some $C_0(\go)$ Hilbert module $\X$.  In that case, define the
$\T$-groupoid
\[ E:= \{(x,U)\mid U\colon \X(s(x))\to \X(r(x)) \;\text{is a unitary
  with}\; \alpha_x = \Ad U \;\text{for}\;x\in G\}\] with
$(x,U)(y,V):=(xy, UV)$.  We show that $\cs(G;E)$ is Morita equivalent
to $\K(\X)\rtimes_{\alpha} G$ for both the maximal and reduced
$\cs$-algebras.

Even if the Dixmier-Douady invariant of $\A$ is not zero, there always
exists an open cover $\U=\sset{U_{i}}$ of $\go$ by pre-compact open
sets such that the restriction of $\A$ to each $U_i$ has zero
Dixmier-Douady invariant.
Let $\uG:=G[\U]$ be the groupoid with unit space $\uGo := \coprod U_i$
obtained as the pullback of $G$ via $\uGo \to \go$.  As a pullback
groupoid, $\uG$ is equivalent to $G$.

The convolution $\cs$-algebras (maximal or reduced) of equivalent
groupoids are Morita equivalent.  We show that the crossed products
(maximal or reduced) $\A\rtimes_\alpha G$ and $\uA\rtimes_\alpha \uG$
are Morita equivalent, if we let $\uA$ be the pullback of $\A$ via
$\uGo\to \go$.

By construction, $\uA$ has Dixmier-Douady invariant zero, and
therefore there is a twist $\uE$ such that $\uA\rtimes_\alpha \uG$ is
Morita equivalent to $\cs(\uG;\uE)$.

A technical difficulty is to provide $\uE$ with the right topology,
and to prove that it satisfies the axioms of a twist.  To this end we
introduce two auxiliary groupoids $\Aut \A$ and $\Iso(\XX)$, whose
construction may be of independent interest.

\subsection{Two Useful Groupoids}
\label{sec:some-usef-topol}
{\emergencystretch=15pt In this section we introduce two groupoids $\Aut \A$ and $\Iso(\XX)$
that are convenient in what follows.  The constructions given here are
valid and may be of use in other contexts, even if $\A$ is not
continuous trace.\par}  \vskip 6pt

Let $p:\A\to\go$ be an upper semicontinuous \cs-bundle over $\go$.  We
define
\begin{equation*}
  \Aut\A:=\set{(u,\alpha,v)\mid\text{$\alpha:A(v)\to A(u)$ is a
      $*$-isomorphism}}. 
\end{equation*}
Then $\Aut\A$ is a groupoid with respect to the natural operations.
\begin{prop}
  \label{prop-auta}
  If $p:\A\to\go$ is an upper semicontinuous \cs-bundle over $\go$,
  then $\Aut\A$ has a Hausdorff topology making it into a topological
  groupoid such that $\sset{(u_{i},\alpha_{i},v_{i})}$ converges to
  $(u,\alpha,v)$ if and only if
  \begin{enumerate}
  \item $u_{i}\to u$,
  \item $v_{i}\to v$ and
  \item if $a_{i}\to a_{0}$ in $\A$ and $p(a_{i})=v_{i}$, then
    $\alpha_{i}(a_{i})\to \alpha(a_{0})$ in $\A$.
  \end{enumerate}
\end{prop}

The proof of Proposition~\ref{prop-auta} is a bit tedious.  In
general, one can not specify a topology simply by specifying a
criteria for nets to converge: the family of convergent nets must
satisfy certain axioms (for example, see
\cite{kel:general}*{Chap.~2}).  Furthermore, it seems useful to
exhibit a bona fide base for our topology.  To avoid distractions,
we'll do this in \S\ref{sec:proof-proposition}.

{\emergencystretch=25pt \begin{prop}
    \label{prop-dy-sys}
    Let $\sset{\alpha_{x}}_{x\in G}$ be a family of $*$-isomorphisms
    $\alpha_{x}:A(s(x))\to A(r(x))$.  Then $(\A,G,\alpha)$ is a
    groupoid dynamical system if and only if $x\mapsto
    ((r(x),\alpha_{x},s(x))$ is a continuous groupoid homomorphism
    $G\to \Aut\A$.
  \end{prop}}
\begin{proof}
  This is an immediate consequence of Proposition~\ref{prop-auta}.
\end{proof}

{\emergencystretch=25pt Now we suppose that $\X$ is a Hilbert
  $C_{0}(\go)$-module.  Then
  \cite{fd:representations1}*{Theorem~II.13.18} implies that there is
  a topology on $\XX=\coprod_{u\in\go}\X(u)$ making $q:\XX\to\go$ into
  a continuous Banach bundle, or in this case a continuous Hilbert
  bundle, such that $\X\cong \sa_{0}(\go,\XX)$.\par}

We let
\begin{equation}
  \label{eq:5}
  \Iso(\XX):=\set{(u,V,v)\mid\text{$V:\X(v)\to\X(u)$ is a unitary}}
\end{equation}
be the groupoid with the obvious operations: $(u,V,v)(v,W,w)=(u,VW,w)$
etc.
\begin{prop}
  \label{prop-iso-top}
  If $\X=\sa_{0}(\go,\XX)$ is a Hilbert $C_{0}(\go)$-module, then
  $\Iso(\XX)$ has a Hausdorff topology making it into a topological
  groupoid such that a net $\sset{(u_{i},V_{i},v_{i})}$ converges to
  $(u,V,v)$ if and only if
  \begin{enumerate}
  \item $u_{i}\to u$,
  \item $v_{i}\to v$ and
  \item if $h_{i}\to h_{0}$ in $\XX$ and $q(h_{i})=v_{i}$, then
    $V_{i}h_{i}\to Vh_{0}$ in $\XX$.
  \end{enumerate}
\end{prop}

The proof is similar to that for Proposition~\ref{prop-auta}, so we
omit it.  Note that the relative topology on $U(\X(u))$ is the strong
operator topology.  In particular, $\Iso(\XX)$ is not locally compact
unless all the fibers $\X(u)$ are finite dimensional.

Now we restrict ourselves to the special case where $\X$ is an $A\sme
C_{0}(\go)$-\ib; i.e., $\X$ is a right Hilbert $C_{0}(\go)$-module and
$A$ is isomorphic to the generalized compact operators $\K(\X)$ on
$\X$ via the left action of $A$.  Then it is not hard to see that
$A\cong \sa_{0}(\go,\A)$ with $A(u)$ identified with the compact
operators $\K(\X(u))$ in such a way that $\lip A(u)<h,k>$ is
identified with the rank-one operator $\theta_{h,k}$ where
$\theta_{h,k}(l)=\rip\C<k,l>h$.  A unitary operator $U:\X(v)\to \X(u)$
induces a $*$-isomorphism $\alpha:=\Ad U:A(v)\to A(u)$ given by
$\alpha(a)\cdot h= U(a\cdot U^{*}h)$, with $a\in A(v), h\in \X(u)$.

\begin{prop}\emergencystretch=25pt
  \label{prop-new-j-map}
  Let $\X=\sa_{0}(\go,\XX)$ and $A=\K(\X)=\sa_{0}(\go,\A)$ be as
  above.  Then there is a short exact sequence of topological
  groupoids
  \begin{equation*}
    \xymatrix{\go\times\T\ar[r]^{\iotau}&\Iso(\XX)\ar[r]^{\uj}&\Aut \A},
  \end{equation*}
  where $\uj(u,U,v) = (u,\Ad U,v)$ is an continuous open surjection,
  and $\iotau(u,z)=(u,zI_{\X(u)},u)$ is a homeomorphism onto the
  kernel of $\uj$.  In particular, given $(u,\alpha,v)\in\Aut\A$ with
  $\alpha=\Ad U$, there is a neighborhood $N$ of $(u,\alpha,v)$ and a
  continuous section $\beta:N\to \Iso(\H)$ for $\uj$ such that
  $\beta(u,\alpha,v)=(u,U,v)$.
\end{prop}

Since establishing some of the assertions in
Proposition~\ref{prop-new-j-map}, such as the continuity and openness
of $\uj$, is a bit technical, we have moved the proof of the
proposition to \S\ref{sec:proof-prop-refpr} so as not to
distract from the matter at hand.

Note that $\Iso(\XX)$ admits a $\T$-action such that
\begin{equation*}
  z\cdot (u,U,v)=\iotau(u,z)(u,U,v)=(u,z U,v)=(u,U,v)\iotau(v,z).
\end{equation*}
Thus, except for the fact that $\Iso(\XX)$ and $\Aut\A$ are generally
not locally compact, we can view $\Iso(\XX)$ as a $\T$-groupoid over
$\Aut \A$.

\subsection{The Construction of the Twist}
\label{sec:construction-twist}

Suppose that $(\A,G,\alpha)$ is a groupoid dynamical system with
$A=\sa_{0}(\go,\A)$ continuous trace.  Then there is an open cover
$\U=\sset{U_{i}}$ of $\go$ by pre-compact open sets such that
$A_{U_{i}}:=\sa_{0}(U_{i},\A)$ is Morita equivalent to
$C_{0}(U_{i})$. (See, for example, \cite{rw:morita}*{Proposition~5.5})
In particular, $A_{U_{i}}$ can be identified with $\K(\X_{i})$ for a
Hilbert $C_{0}(U_{i})$-module $\X_{i}=\sa_{0}(U_{i},\XX_{i})$.

Let $\uG:=G[\U]$ be the pullback groupoid whose unit space is the
disjoint union $\uGo:=\coprod U_i=\set{(i,u):x\in U_{i}}$; that is,
\begin{equation*}
  G[\U]:=\set{(i,x,j):\text{$r(x)\in U_{i}$ and $s(x)\in U_{j}$}}
\end{equation*}
with $(i,x,j)(j,y,k)=(i,xy,k)$ etc.  As is well-known, any groupoid of
the form $G[\U]$ for an open cover $\U$ of $\go$ is equivalent to $G$
via $Z=\coprod_{i}G_{U_{i}}$.




Let $\uA$ be the pull-back via the obvious map $\psi:\uGo\to\go$:
\begin{equation*}
  \uA=\set{(i,u,b):\text{$u\in U_{i}$ and $b\in A(u)$}}.
\end{equation*}
Note that $\psi^{*}A\cong \sa_{0}(\uGo,\uA)$ by
\cite{raewil:tams85}*{Proposition~1.3}.  It is not hard to check that
$\psi^{*}A$ is Morita equivalent to $C_{0}(\uGo)$ via
$\X=\sa_{0}(\uGo,\XX)$ where $\XX=\coprod
\XX_{i}=\set{(i,u,h):\text{$u\in U_{i}$ and $h\in \X_{i}(u)$}}$.

There is an action $\alphau:\uG\to \Aut\uA$ of $\uG$ on $\uA$ defined
by
\begin{equation*}
  \alphau(k,x,l)=\bigl(k,r(x),\alpha_{x},s(x),l\bigr),
\end{equation*}
and it is straightforward to see that $\alphau$ is
continuous.\footnote{In fact $(\uA,\uG,\alphau)$ is a dynamical system
  whose class in the Brauer group $\operatorname{Br}(\uG)$ matches up
  with that of $(\A,G,\alpha)$ in $\operatorname{Br}(G)$ under the
  isomorphism of $\operatorname{Br}(G)$ with $\operatorname{Br}(\uG)$
  (see \cite{kmrw:ajm98}*{Theorem~4.1}).}

Let
\begin{equation*}
  \xymatrix{\uGo\times\T\ar[r]^{\iotau}&\Iso(\XX)\ar[r]^{\uj}&\Aut \uA},
\end{equation*}
be the short exact sequence coming from
Proposition~\ref{prop-new-j-map} applied to $\psi^{*}A$, and form the
pull-back
\[ \uE:=\set{(h,\gamma)\in\Iso(\XX)\times \uG: \uj(h)=\alphau(\gamma)}
\]
(equipped with the relative product topology) that completes the
commutative diagram
\begin{equation*}
  \xymatrix{\tE\ar[r]^{j} \ar[d]_{\pr_{1}} & \uG\ar[d]^{\alphau} \\
    \Iso(\XX)\ar[r]_{\uj}&\Aut\uA \;.}
\end{equation*}
A description of $\uE$ that is easier to work with is
\begin{multline*}
  \uE=\{\,(k,x,V,l): x\in G, r(x)\in U_k, s(x)\in U_l, \\
  \text{and $V\colon \X(s(x))\to \X(r(x))$ is a unitary with
    $\alpha_{x}=\Ad V$}\,\}
\end{multline*}
with multiplication $(k,x,V,l)(l,y,W,m)=(k,xy,VW,m)$.

Consideration of the commutative diagram that defines $\uE$ shows that
$j$ is an open, continuous surjection with kernel identified with
$\iota(\uGo\times\T)$ (for the obvious map $\iota$), and that $\tE$ is
a Hausdorff topological groupoid admitting a central $\T$-action.  In
fact, $\tE$ is a principal $\T$-bundle over the locally compact space
$\uG$.  Hence $\tE$ is a \emph{locally compact} $\T$-groupoid over
$\uG$.

\section{The Main Theorem}
\label{sec:main-theorem}

\begin{thm}
  \label{thm-main}
  Suppose that $G$ is a second countable locally compact Hausdorff
  groupoid with Haar system $\sset{\lambda^{u}}_{u\in\go}$ and that
  $(\A,G,\alpha)$ is a dynamical system with $A:=\sa_{0}(\go,\A)$
  continuous trace.  Let $\uE$ be the $\T$-groupoid over the
  equivalent groupoid $\uG$ constructed in the previous section.  Then
  $\A\rtimes_{\alpha}G$ is Morita equivalent to $\cs(\uG;\uE)$, and
  $\A\rtimes_{\alpha,r}G$ is Morita equivalent to $\cs_{r}(\uG,\uE)$.
\end{thm}

We are going to realize both $\A\rtimes_{\alpha}G$ and $\cs(\uG;\uE)$
as Fell bundle \cs-algebras and then observe that the underlying Fell
bundles are equivalent (as in \cite{muhwil:dm08}*{Definition~6.1}).
Then Theorem~\ref{thm-main} will follow from the Fell Bundle
Equivalence Theorem \cite{muhwil:dm08}*{Theorem~6.4} and the
observation that the Morita equivalence descends to the reduced
algebras \cite{simwil:nyjm13}*{Theorem~14}.

Recall from Example~\ref{ex-twist-fell} that $\cs(\uG;\uE)$ is
isomorphic to $\cs(G,\CC)$ where $\CC$ is the line bundle
$\T\backslash (\uE,\C)$ where $z\cdot (\ue, \lambda)=(z
\ue,z\lambda)$.  Thus
\begin{equation}
  \label{eq:7}
  \CC=\set{[k,x,U,\lambda,l]:\text{$(k,x,U,l)\in\uE$, $\lambda\in\C$
      and $\Ad U=\alpha_{x}$}}, 
\end{equation}
and we have identified $(k,x,zU,\lambda,l)$ with $(i,x,U,\bar
z\lambda,j)$ for all $z\in \T$. In particular, the map $[k,x,U,
\lambda,l]\mapsto \lambda U^{*}$ is well defined on $\CC$.

On the other hand, as in Example~\ref{ex-cross-prod},
$\A\rtimes_{\alpha}G$ is isomorphic to $\cs(G,\B)$ where $\B=s^{*}\A$.

Let $Z=\coprod_{i}G_{U_{i}}$ be the $(G,\uG)$-equivalence from
\S\ref{sec:phill-raeb-appr}.  We'll show that
\begin{equation}
  \label{eq:8}
  \E=s^{*}\XX =\set{(i,x,h):\text{$s(x)\in U_{i}$ and $h\in
      \X_{i}$}} ,
\end{equation}
viewed as a bundle over $Z$, is the desired equivalence between $\B$
and $\CC$.

To start with, we need actions of $\B$ and $\CC$ on $\E$ satisfying
the properties (a), (b) and (c) laid out just prior to
\cite{muhwil:dm08}*{Definition~6.1}:\footnote{Sadly, condition~(c)
  there should read: $\|b\cdot e\|\le \|b\|\|e\|$ (and not $\|b\cdot
  e\| = \|b\|\|e\|$).\label{fn:4}} the left $\B$-action is given by
\begin{equation}
  \label{eq:9}
  (x,a)\cdot (i,y,h):= \bigl(i,xy,\alpha_{y}^{-1}(a)\cdot h\bigr),
\end{equation}
and the right $\CC$-action by
\begin{equation}
  \label{eq:10}
  (i,y,h)\cdot [i,z,U,\lambda,j] := (j,yz,\lambda U^{*}h).
\end{equation}
Continuity follows from the characterization of the topologies on
$\Aut\A$ and $\Iso(\XX)$ in \S\ref{sec:some-usef-topol} and
Lemma~\ref{lem-product}.


Next we check that the axioms of \cite{muhwil:dm08}*{Definition~6.1}
hold.  To see that the actions commute we observe that on the one
hand,
\begin{align*}
  (x,a)\cdot \bigl((i,y,h)\cdot [i,z,U,\lambda,j]\bigr) &= (x,z)\cdot
  \bigl( j, yz, \lambda U^{*}h\bigr) \\
  &= (j,xyz,\alpha_{yz}^{-1}(a)\cdot \lambda U^{*}h).
\end{align*}
On the other hand,
\begin{align*}
  \bigl((x,a)\cdot (i,y,h)\bigr) \cdot [i,z,U,\lambda,j] & =
  \bigl((i,xy,\alpha_{y}^{-1}(a)\cdot h\bigr) \cdot  [i,z,U,\lambda,j] \\
  &= \bigl(j,xyz,\lambda U^{*}\alpha_{y}^{-1}(a)\cdot h\bigr) \\
  &= \bigl(j,xyz, \lambda
  \alpha_{z}^{-1}\bigl(\alpha_{y}^{-1}(a)\bigr)\cdot U^{*}h\bigr).
\end{align*}
Hence the actions do commute.

To define inner products, we proceed as follows.
\begin{equation}
  \label{eq:11}
  \blip\star<{(i,x,h)},{(i,y,k)}>= \bigl(xy^{-1},\alpha_{y}(\lip
  A(s(y))<h,k>)\bigr), 
\end{equation}
and
\begin{equation}
  \label{eq:12}
  \brip\star<{(i,y,k)},{(j,z,l)}>= \bigl[
  i,y^{-1}z,U,\rip\C<k,Ul>,j\bigr] ,
\end{equation}
where the definition of $\CC$ allows us to choose any unitary
$U:\X(i,s(z))\to\X(j,s(y))$ implementing~$\alpha_{y^{-1}z}$.

To see that the actions and inner products play nice together (as in
\cite{muhwil:dm08}*{Definition~6.1(c)(iv)}), we proceed as follows.
On the one hand,
\begin{align*}
  \blip\star<{(i,x,h)},{(i,y,k)}>\cdot (j,z,l) &=
  \bigl((xy^{-1},\alpha_{y} (\lip A(s(y))<h,k>) \bigr)\cdot (j,z,l) \\
  &=\bigl(j,xy^{-1}z,\alpha_{z}^{-1}\bigl(\alpha_{y}(\lip
  A(s(y))<h,k>)\bigr)
  \cdot l\bigr) \\
  &= \bigl(j,xy^{-1}z,\alpha_{z^{-1}y}(\lip A(s(y))<h,k>)\cdot l\bigr)
  .
\end{align*}
On the other hand,
\begin{align*}
  (i,x,h)\cdot \brip\star<{(i,y,k)},{(j,z,l)}> &= (i,x,h)\cdot \bigl[
  i, y^{-1}z ,U ,\rip\C<k,Ul>,j\bigr] \\
  &= \bigl(j,xy^{-1}z,\rip\C<k,Ul>U^{*}(h)\bigr) \\
  &= \bigl(j,xy^{-1}z,U^{*}\bigl(\rip\C<k,Ul>h\bigr)\bigr)\\
  &= \bigl((j,xy^{-1}z, U^{*}(\lip A(s(y)<h,k>\cdot Ul)\bigr) \\
  &= \bigl(j,xy^{-1}z,\alpha_{z^{-1}y}(\lip A(s(y))<h,k>)\cdot
  l\bigr).
\end{align*}
Thus
\begin{equation}
  \label{eq:13}
  \blip\star<{(i,x,h)},{(i,y,k}>\cdot (j,z,l) =  (i,x,h)\cdot
  \brip\star<{(i,y,k)},{(j,z,l)}> 
\end{equation}
as required.  Checking the rest of
\cite{muhwil:dm08}*{Definition~6.1(c)} is more straightforward.

It remains only to check that $E(i,x)$, equipped with the given
actions and inner products, is a $B(r(x))\sme C(i,s(x))$-\ib.  But
$C(i,s(x)) =\set{[i,s(x),zI,\lambda,i]}$ can be identified with $\C$
via the map $[i,s(x),zI,\lambda,i]\mapsto \bar z\lambda$, and
$B(r(x))$ with $A(r(x))$.  Then $E(i,x)$ is isomorphic to
${}_{\alpha_{x}}\X(i,x)$.\footnote{Recall that if $\mathsf{X}$ is an
  $A\sme B$-\ib, and if $\theta:A\to C$ is an isomorphism, then
  ${}_{\theta}\mathsf{X}$ is the $C\sme B$-\ib, where the left inner
  product is given by $\lip C<x,y>=\theta(\lip A<x,y>)$ and the left
  $C$-action is given by $c\cdot x:= \theta^{-1}(c)\cdot x$.  (Of
  course the right Hilbert-$B$ module structure is as
  before.)\label{fn:ib}}

Thus axioms (a), (b) and (c) of \cite{muhwil:dm08}*{Definition~6.1}
are satisfied.  Thus $\E$ is an equivalence between $\B$ and $\CC$.
This completes the proof of Theorem~\ref{thm-main}.

\section{A Partial Converse}
\label{sec:partial-converse}


In this section, we briefly outline a proof of the following.
\begin{prop}
  \label{prop-converse}
  Suppose that $E$ is a $\T$-groupoid over a second countable locally
  compact groupoid $G$ with a Haar system
  $\sset{\lambda^{u}}_{u\in\go}$.  Then there is a Hilbert
  $C_{0}(\go)$-module $\X$ and a $G$ action $\alpha$ on $A=\K(\X)$ so
  that $A\rtimes_{\alpha}G$ is Morita equivalent to $\cs(G;E)$, and
  $A\rtimes_{\alpha,r}G$ is Morita equivalent to $\cs_{r}(G;E)$.
\end{prop}

Let $\CC$ be the Fell bundle associated to $E$ as in
Example~\ref{ex-twist-fell}.  Recall that $C_{c}(G;E)$ consists of
functions in $C_{c}(E)$ that satisfy $f(ze)= zf(e)$.  Then $\ccge$
becomes a pre-Hilbert $C_{0}(\go)$-module via the pre-inner product
\begin{equation*}
  \rip C_{0}(\go)<f,g> (u) =\int_{G} \overline{f(e)} g(e)
  \,d\lambda_{u}(j(e)) =\int_{G} \overline{f(e^{-1})} g(e^{-1})
  \,d\lambda^{u}(j(e)) .
\end{equation*}
The completion, $\X$, is the section algebra $\X=\sa_{0}(\go,\XX)$ of
a continuous Hilbert bundle $\XX$ over $\go$.  The fibre $\X(u)$ over
$u\in \go$ is the Hilbert space which is the completion of $\X_{0}(u)
=\set{\phi\in C_{c}(E): f(ze)=zf(e)}$ with respect to the pre-inner
product induced by $\rip C_{0}(\go)<\cdot,\cdot>$. (Of course, this is
a Hilbert space with an inner product conjugate linear in the first
variable.)  The algebra of generalized compacts,
$\K(\X)$, is Morita equivalent to $ C_{0}(\go)$.  Thus $\K(\X)$ is a
continuous-trace \cs-algebra with spectrum $\go$ and trivial
Dixmier-Douady class \cite{rw:morita}*{\S5.3}.  Furthermore,
$\K(\X)\cong \sa_{0}(\go, \KK)$ for a suitable (continuous) \cs-bundle
$\KK$ over $\go$.

For each $e\in E$, define a unitary $u(e):\X(s(e))\to\X(r(e))$ by
\begin{equation*}
  (u(e))(f)(e'):=f(e'e).
\end{equation*}
Since $u(ze)=z u(e)$, $\Ad u(e):\K(\X)(s(e))\to\K(\X)(r(e))$ depends
only on $j(e)\in G$, we get an action $\alpha=\sset{\alpha_{x}}_{x\in
  G}$ of $G$ on $\KK$ by
\begin{equation*}
  \alpha_{j(e)}(T)=\Ad u(e)\circ T.
\end{equation*}
(Just for the record, $\Ad u(e)(T)=u(e)Tu(e^{-1})$.)

It is not hard to see that this gives a continuous action of $G$ on
$\KK$.  Hence $(\KK,G,\alpha)$ is a groupoid dynamical system and we
can form the groupoid crossed product: $\KK\rtimes_{\alpha}G$.  As in
Example~\ref{ex-cross-prod}, this crossed product is isomorphic to the
Fell bundle \cs-algebra $\cs(G,\B)$ where $\B=s^{*}\KK=\set{(x,T)\in
  G\times \KK:T\in \K\bigl(\X(s(x))\bigr)}$, with
$(x,T)(y,S)=(xy,\alpha_{y}^{-1}(T)S)$ and
$(x,T)^{*}=(x^{-1},\alpha_{x}(T^{*}))$.

To prove Proposition~\ref{prop-converse}, we will treat $G$ as a
$(G,G)$-equivalence and show that $\E=s^{*}\XX=\set{(x,h)\in
  G\times\XX: h\in\X(s(x))}$, viewed as a bundle over $G$, is an
equivalence between $\B$ and $\CC$.  (This will suffice by the
Equivalence Theorem \cite{muhwil:dm08}*{Theorem~6.4} and
\cite{simwil:nyjm13}*{Theorem~14}.)

Now we proceed as in the previous section.  We need actions of $\B$
and $\CC$ on $\E$ satisfying the properties (a), (b)~and (c) laid out
just prior to \cite{muhwil:dm08}*{Definition~6.1}.
We define
\begin{equation*}
  (x,T)\cdot (y,h):= \bigl(xy,\alpha_{y}^{-1}(T)(h)\bigr) \quad\text{and} \quad
  (y,h)\cdot [e,\lambda]:= \bigl(yj(e),\lambda u(e^{-1})(h)\bigr)
\end{equation*}
for $(x,T)\in\B$, $(y,h)\in\E$ and $[e,\lambda]\in \CC$.\footnote{To
  see that the $\CC$-action is well-defined, keep in mind that
  $u(z\cdot e)=z u(e)$.}  The algebraic properties are easily checked,
and continuity is not so hard to check.

Then we need to check that the actions commute.  On the one hand,
\begin{align*}
  \bigl((x,T)\cdot (y,h)\bigr) \cdot [e,\lambda] &=
  \bigl(xy,\alpha_{y}^{-1}(T)(h)\bigr) \cdot [e,\lambda] \\
  &= \bigl(xyj(e),\lambda u(e^{-1})
  \bigl(\alpha_{y}^{-1}(T)(h)\bigr)\bigr) .
\end{align*}
On the other hand, keeping in mind that $\alpha_{j(e)}=\Ad u(e)$, we
have
\begin{align*}
  (x,T)\cdot \bigl((y,h)\cdot [e,\lambda]\bigr) &= (x,T)\cdot
  \bigl(yj(e), \lambda u(e^{-1})(h)\bigr) \\
  &= \bigl(xyj(e),\lambda
  \alpha_{yj(e)}^{-1}(T)\bigl(u(e^{-1})(h)\bigr)\bigr) \\
  &= \bigl(xyj(e), \lambda u(e^{-1})\alpha_{y}^{-1}(T)(h)\bigr).
\end{align*}
Thus we have verified Definition~6.1(a).

Next we define $\lip\B<\cdot,\cdot>$ on $\E*_{s}\E$ by
\begin{equation*}
  \blip\B<(x,h),(y,k)> = \bigl(xy^{-1}, \alpha_{y}(\theta_{h,k})\bigr)
\end{equation*}
where $\theta_{h,k}:\X(s(y))\to\X(s(y))$ is the rank-one operator
$\theta_{h,k}(l)=\rip\C<k,l>h$, where $\rip\C<\cdot,\cdot>=\rip
C_{0}(\go)<\cdot,\cdot>(u)$.  Note that if $y=j(e)$, then
$\alpha_{y}(\theta_{h,k})=\theta_{u(e)h,u(e)k}$.

We define $\rip\CC<\cdot,\cdot>$ on $\E*_{r}\E$ by
\begin{equation*}
  \brip\CC<(j(f),k),(j(g),l)> = \bigl[f^{-1}g, \brip C_{0}(\go)
  <u(f)k,u(g)l>(r(f)) \bigr].
\end{equation*}
After checking that the above is actually well-defined, it is not too
difficult to see that $\lip\B<\cdot,\cdot>$ and $\rip\CC<\cdot,\cdot>$
satisfy (i)--(iv) of part~(b) of the definition.  For example, to
verify~(iv), consider on the one hand:
\begin{align*}
  \blip\B<(x,h),(y,k)> \cdot (z,l) & = \bigl(xy^{-1},
  \alpha_{y}(\theta_{h,k})\bigr) \cdot (z,l) \\
  &= \bigl(xy^{-1}z, \alpha_{z^{-1}y}(\theta_{h,k})(l)\bigr).
\end{align*}
On the other hand, supposing that $j(f)=y$ and $j(g)=z$, we have
\begin{align*}
  (x,h)\cdot \brip\CC<(j(f),k),(j(g),l)> & = (x,h) \cdot \bigl[
  f^{-1}g, \brip C_{0}(\go)<u(f)k,u(g)l>(r(f)) \bigr] \\
  &= \bigl( xy^{-1}z, \brip C_{0}(\go)<u(f)k,u(g)l>(r(f))
  u(g^{-1}f)(h)
  \bigr ) \\
  &= \bigl ( xy^{-1}z, \brip C_{0}(\go) < u(g^{-1}f)k,l>(s(g))
  u(g^{-1}f)(h) \bigr ) \\
  &= \bigl(xy^{-1}z, \theta_{u(g^{-1}f)h,u(g^{-1}f)k}(l)\bigr) \\
  &= \bigl( xy^{-1}z, \alpha_{z^{-1}y}(\theta_{h,k})(l) \bigr).
\end{align*}
Thus (iv) holds.

We still need to verify part~(c) of the Definition; that is, we need
to see that
\begin{equation*}
  E(x)=\set{(x,h):h\in \X(s(x))}
\end{equation*}
is a $B(r(x))\sme C(s(x))$-\ib\ with respect to the inherited
operations. First note that $C(s(x))=\set{[(\iota(s(x),z),\lambda]}$
is easily identified with $\C $ via the map
$[(\iota(s(x),z),\lambda]\mapsto \bar z\lambda$.  On the other hand,
$B(r(x))$ is easily identified with $\K(\X(r(x)))$.  Then it is not
hard to check that $E(x)$ is the $\K(\X(r(x)))\sme \C$-\ib\ given by
${}_{\alpha_{x}}\X(s(x))$.\footnote{See footnote~\ref{fn:ib} at the
  end of \S\ref{sec:main-theorem}.}

Thus axioms (a), (b) and (c) of \cite{muhwil:dm08}*{Definition~6.1}
are satisfied and $\E$ is the desired equivalence.  This completes the
proof of Proposition~\ref{prop-converse}.

\section{Proof of Proposition~\ref{prop-auta}}
\label{sec:proof-proposition}

Recall that if $X$ is a set and $\rho$ is a collection of subsets of
$X$ than cover $X$, then the collection of finite intersections of
elements of $\rho$ form a basis for a topology $\tau$ on $X$.  In this
case, we say $\rho$ is a \emph{subbasis} for $\tau$.

If $a,b\in A$, $U$ and $V$ are open sets in $\go$ and $\epsilon>0$,
then let
\begin{equation*}
  \W(U,a,b,V,\epsilon)=\set{(u,\alpha,v):\text{$u\in U$, $v\in V$ and
      $\|\alpha(a(v))-b(u)\|<\epsilon$}} .
\end{equation*}
Of course, some of the $\W(U,a,b,V,\epsilon)$ might be empty, but if
$(u,\alpha,v)\in\Aut\A$, for any $a\in A$, there is a $b\in A$, such
that $b(v)=\alpha(a(u))$.  Then $(u,\alpha,v)\in \W(U,a,b,V,\epsilon)$
for any $U$, $V$ and $\epsilon$.  Hence the collection of all 
$\W(U,a,b,V,\epsilon)$
cover $\Aut\A$, and form
a subbasis for a topology $\tau$ on $\Aut\A$.

\begin{lemma}
  \label{lem-iff}
  We have $(u_{i},\alpha_{i},v_{i})\to (u,\alpha,v)$ in
  $(\Aut\A,\tau)$ if and only if \textup{(a)}, \textup{(b)} and \textup{(c)} of
  Proposition~\ref{prop-auta} hold.
\end{lemma}
\begin{proof}
Suppose that $(u_{i},\alpha_{i},v_{i})\to (u,\alpha,v)$.
  Fix $a,b\in A$ such that $b(u)=\alpha(a(v))$.  Then for any
  $\epsilon>0$ and any neighborhoods $U$ of $u$ and $V$ of $v$, we
  have that $\W(U,a,b,V,\epsilon)$ is a neighborhood of
  $(u,\alpha,v)$.  Hence (a)~and (b) hold.

   There is an $i_{0}$ such that $i\ge i_{0}$ implies
  $(u_{i},\alpha_{i},v_{i})\in \W(U,a,b,V,\epsilon)$.  Hence if $i\ge
  i_{0}$, we have
  \begin{equation*}
    \|\alpha_{i}(a(v_{i}))-b(u_{i})\|<\epsilon.
  \end{equation*}
Since $b(u_{i})\to b(u)$ in $\A$ and $\epsilon>0$ is arbitrary,
\cite{wil:crossed}*{Proposition~C.20} implies that
$\alpha_{i}(a(v_{i}))\to \alpha(a(v))$.  Now if $a_{i}\to a$ in $\A$
with $p(a_{i})=v_{i}$, let $a\in A$ be such that $a(v)=a_{0}$.  Then
by the above, $\alpha_{i}(a(v_{i}))\to \alpha(a_{0})$ in $\A$, while
\begin{equation*}
  \|\alpha_{i}(a_{i})-\alpha_{i}(a(v_{i}))\|=\|a_{i}-a(v_{i})\|\to 0,
\end{equation*}
since isomorphisms are isometric and $a$ is a section.  It follows
from another application of \cite{wil:crossed}*{Proposition~C.20} that
$\alpha_{i}(a_{i})\to\alpha(a_{0})$.  This proves (c).

Conversely, suppose that $\sset{(u_{i},\alpha_{i},v_{i})}$ is a net in
$\Aut\A$ satisfying (a), (b)~and (c).  Let $\W(U,a,b,V,\epsilon)$ be a
subbasic open neighborhood of $(u,\alpha,v)$.  It will suffice to see
that $\sset{(u_{i},\alpha_{i},v_{i})}$ is eventually in
$\W(U,a,b,V,\epsilon)$.  Since (a)~and (b) hold, we can assume that
$u_{i}\in U$ and $v_{i}\in V$.  Moreover,
\begin{equation*}
  \|\alpha(a(v))-b(u)\|<\epsilon.
\end{equation*}
But $\set{a\in\A:\|a\|<\epsilon}$ is open and
\begin{equation*}
  \alpha_{i}(a(v_{i}))-b(u_{i})\to \alpha(a(v))-b(u)
\end{equation*}
in view of assumption~(c).  Thus we eventually have
\begin{equation*}
  \|\alpha_{i}(a(v_{i}))-b(u_{i})\|<\epsilon.
\end{equation*}
That is, $\sset{(u_{i},\alpha_{i},v_{i})}$ is eventually in
$\W(U,a,b,V,\epsilon)$ as required.
\end{proof}

\begin{remark}
  \label{rem-point-norm}
  It follows immediately that the relative topology on $\Aut A(u)$ is
the so-called ``point-norm'' topology (see
\cite{rw:morita}*{Lemma~7.18}).
\end{remark}
  
More generally, we have the following
immediate corollary.
\begin{cor}
  \label{cor-fibre-top}
  A net $\sset{(u,\alpha_{i},v)}$ converges to $(u,\alpha,v)$ in
  $\Aut\A$ if and only if $\alpha_{i}(b)\to \alpha(b)$ for all $b\in
  A(v)$. 
\end{cor}

\begin{lemma}
  \label{lem-t2}
  The topology $\tau$ on $\Aut\A$ is Hausdorff.
\end{lemma}
\begin{remark}
  \label{rem-eveninusc}
  At this point, we are \emph{not} assuming that $\A$ is a continuous
  bundle.  Hence the topology on $\A$ need not be Hausdorff in general
  \cite{wil:crossed}*{Example~C.27}. 
\end{remark}

\begin{proof}
  The existence of neighborhoods $\W(U,a,b,V,\epsilon)$ of
  $(u,\alpha,v)$ with $U$ and $V$ arbitrary implies that we only have
  to show that we can separate $(u,\alpha,v)$ and $(u,\beta,v)$ with
  $\alpha\not=\beta$.  But this is not hard in view of
  Corollary~\ref{cor-fibre-top}. 
\end{proof}

\begin{proof}[Proof of Proposition~\ref{prop-auta}]
  Since $\tau$ is Hausdorff by Lemma~\ref{lem-t2}, we just need to see
  that the groupoid operations are continuous.  But since
  $*$-isomorphisms are isometric, it follows that $(u,\alpha,v)\in
  \W(U,a,b,V,\epsilon)$ if and only if $(v,\alpha^{-1},u)\in
  \W(V,b,a,U,\epsilon)$.  Hence, inversion is continuous.

  To see that multiplication is continuous, consider nets
  \begin{equation*}
    (u_{i},\alpha_{i},v_{i})\to (u,\alpha,v)\quad\text{and} \quad
    (v_{i},\beta_{i},w_{i})\to (v,\beta,w).
  \end{equation*}
We just need to see that $(u_{i},\alpha_{i}\circ\beta_{i},w_{i})\to
(u,\alpha\circ\beta, w)$.  But if $a_{i}\to a_{0}$ in $\A$ with
$p(a_{i})=u_{i}$, then Lemma~\ref{lem-iff} implies that
$\beta_{i}(a_{i})\to \beta(a_{0})$ and hence that
$\alpha_{i}(\beta_{i}(a_{i}))\to \alpha(\beta(a_{0}))$.  By
Lemma~\ref{lem-iff}, this suffices.
\end{proof}

\section{Proof of Proposition~\ref{prop-new-j-map}}
\label{sec:proof-prop-refpr}

Recall that $\X(u)$ is an $A(u)\sme \C$-\ib\ with
\begin{equation*}
  \lip
A(u)<a(u),b(u)>=\lip A<a,b>(u)\quad\text{for all $a,b\in A$.}
\end{equation*}
\begin{lemma}
  \label{lem-j-cts}
  The map $\uj:\Iso(\XX)\to\Aut\A$ is continuous.
\end{lemma}
\begin{proof}
  Suppose that $(u_{i},V _{i},v_{i})\to (u,V,v)$ in $\Iso(\XX)$ and
  that $a_{i}\to a_{0}$ in $\A$ with $p(a_{i})=v_{i}$.  It will
  suffice to see that $(\Ad V_{i})(a_{i})\to (\Ad V)(a_{0})$ in $\A$.

  Given $\epsilon>0$, there are elements $c_{j},d_{j}\in A$ such that
  \begin{equation*}
    \bigl\|\sum_{j=1}^{n} \lip A(v)<c_{j}(v),d_{j}(v)> -a_{0}\bigr\|<\epsilon.
  \end{equation*}
  Since $\Ad V\bigl(\lip A(v)<c,d>\bigr)=\lip A(u)<Vc,Vd>$ and since
  $\Ad V$ is isometric,
  \begin{equation*}
    \bigl\|\sum_{j} \lip A(u)<Vc_{j}(v), Vd_{j}(v)>-\Ad
    V(a_{0})\bigr\|<\epsilon. 
  \end{equation*}
  But $\bigl(\sum_{j} \lip A(v_{i})<c_{j}(v_{i}),d_{j}(v_{i})> -a_{i
  }\bigr)$ converges to $\bigl(\sum_{j} \lip A(v)<c_{j}(v),d_{j}(v)>
  -a_{0}\bigr)$ in $\A$.  Hence we eventually have
  \begin{equation*}
    \bigl\|\sum_{j}\lip
    A(v_{i})<c_{j}(v_{i}),d_{j}(v_{i})>-a_{i}\bigr\|<\epsilon. 
  \end{equation*}
  As above, this means we eventually have
  \begin{equation*}
    \bigl\|\sum_{j} \lip A(u_{i})<V_{i}c_{j}(v_{i}), V_{i}d_{j}(v_{i}) >
    -\Ad V_{i}(a_{i}) \bigr\|<\epsilon.
  \end{equation*}
  Since we certainly have $\sum_{j}\lip
  A(u_{i})<V_{i}c_{j}(v_{i}),V_{i}d_{j}(v_{i})>\to \sum_{j}\lip A(u)
  <Vc_{j}(v),Vd_{j}(v)>$ by assumption,
  \cite{wil:crossed}*{Proposition~C.20} implies that $\Ad
  V_{i}(a_{i})\to \Ad V(a_{0})$ as required.
\end{proof}

The next result is classical.  We sketch the proof as the construction
will be required in the proof of the proposition.
\begin{lemma}
  \label{lem-prop-1.6}
  Given a $*$-isomorphism $\alpha:A(v)\to A(u)$, there is a unitary
  $U:\X(v)\to\X(u)$, determined up to a unimodular scalar, such that
  $\alpha=\Ad U$.  Consequently, $\uj$ is surjective with kernel the
  image of $\iotau$.
\end{lemma}
\begin{proof}[Sketch of the Proof]
  The only nontrivial bit is the construction of $U$, and for this we
  follow the proof of \cite{rw:morita}*{Proposition~1.6}.

  Let $e$ be a unit vector in $\X(v)$.  Then $p:=\lip A(v)<e,e>$ is a
  minimal projection in $A(v)$.  Hence $\alpha(p)$ is a minimal
  projection in $A(u)$ and is of the form $\lip A(u)<f,f>$ for any
  unit vector $f$ in the range of $\alpha(p)$.  Then we define
  $U:\X(v)\to \X(u)$ by
  \begin{equation}
    \label{eq:4}
    Uh:=\alpha(\lip A(v)<h,e>)\cdot f.
  \end{equation}
  A computation show that $U$ preserves inner products and if $g\in
  \X(u)$ we have $Uh=g$ with $h=\alpha^{-1}(\lip A(u)<g,f>)e$.  Hence
  $U$ is a unitary.  Another computation shows that $\alpha(T)\cdot
  Uh=U(T\cdot h)$.  Hence $\alpha=\Ad U$ as required.
\end{proof}

By \cite{wil:crossed}*{Proposition~1.15}, the openness of $\uj$ will
follow from the cross section assertion.  Since $\iotau$ is clearly a
homeomorphism onto its range, we can complete the proof of
Proposition~\ref{prop-new-j-map} by proving the cross section result.

\begin{prop}
  \label{prop-cross-section}
  Given $(u_{0},\alpha_{0},v_{0})\in\Aut\A$ with $\alpha_{0}=\Ad
  V_{0}$, there is an open neighborhood $N$ of
  $(u_{0},\alpha_{0},v_{0})$ and a continuous map
  $\beta:N\to\Iso(\XX)$ such that $\uj\circ \beta=\id_{N}$ and such
  that $\beta(u_{0},\alpha_{0},v_{0})=(u_{0},V_{0},v_{0})$.
\end{prop}

We'll need the following lemma.
\begin{lemma}
  \label{lem-product}
  The action of $\A$ on $\XX$ is continuous in that if $a_{i}\to
  a_{0}$ in $\A$ and $h_{i}\to h_{0}$ in $\XX$ with
  $p(a_{i})=v_{i}=q(h_{i})$, then $a_{i}\cdot h_{i}\to a_{0}\cdot
  h_{0}$ in $\XX$.
\end{lemma}
\begin{proof}
  Fix $a\in A$ and $h\in \X$ such that $a(v_{0})=a_{0}$ and
  $h(v_{0})=h_{0}$.  Since $\|a_{i}-a(v_{i})\|\to 0$ and
  $\|h_{i}-h(v_{i})\|\to 0$, it follows that $\|a_{i}\cdot
  h_{i}-a(v_{i})\cdot h(v_{i})\|\to 0$.  Since $a(v_{i})\cdot
  h(v_{i})=a\cdot h(v_{i})$, we have $a(v_{i})\cdot h(v_{i})\to
  a_{0}\cdot h_{0}$ in $\XX$.  Hence $a_{i}\cdot h_{i}\to a_{0}\cdot
  h_{0}$ by \cite{wil:crossed}*{Proposition~C.20}.
\end{proof}

\begin{proof}[Proof of Proposition~\ref{prop-cross-section}]
  We follow the constructions in Lemma~\ref{lem-prop-1.6}.  Thus if
  $e_{0}$ is a unit vector in $\X(v_{0})$ and $c_{0}\in\X(u_{0})$ is a
  unit vector in the range of $\alpha_{0}(\lip
  A(v_{0})<e_{0},e_{0}>)$, then $\alpha_{0}=\Ad V$ for $V$ defined by
  $V(h)=\alpha_{0}(\lip A(v_{0})<h,e_{0}>)c_{0}$. We can modify
  $c_{0}$ by a unimodular scalar so that $V=V_{0}$.

  Let $e\in\X$ be such that $e(v_{0})=e_{0}$.  Since $v\mapsto
  \|e(v)\|$ is continuous and $\phi\cdot e(v)=\phi(v)e(v)$ defines an
  element of $\X$ for any $\phi\in C_{c}(\go)$, we can assume that
  there is a neighborhood $V$ of $v_{0}$ such that $\|e(v)\|=1$ for
  all $v\in V$.  Similarly, choose $c\in X$ such that $c(u_{0})=c_{0}$
  such that there is a neighborhood $U$ of $u_{0}$ such that
  $\|c(u)\|=1$ for all $u\in U$.  Let $a:=\lip A<e,e>$ and $b:=\lip
  A<c,c>$ be the corresponding local rank-one projection fields.
  Then,
  \begin{equation*}
    N=\W(U,a,b,V,\frac14)=\set{(u,\alpha,v):\text{$u\in U$, $v\in V$ and
        $\|\alpha(a(v))-b(u)\|<\frac14$}}
  \end{equation*}
  is a basic open neighborhood of $(u_{0},\alpha_{0},v_{0})$ in $\Aut
  \A$ (see the proof of Proposition~\ref{prop-auta}).

  If $(u,\alpha,v)\in N$, then
  \begin{equation*}
    \|\alpha(a(v))c(u)-c(u)\|=\|\alpha(a(v))c(u)-b(u)c(u)\|\le\frac14.
  \end{equation*}
  In particular, $\alpha(a(v))c(u)\not=0$ and
  \begin{equation*}
    f(u,\alpha,v):= \|\alpha(a(v))c(u)\|^{-1} \alpha(a(v))c(u)
  \end{equation*}
  is a unit vector in the range of the minimal projection
  $\alpha(a(v))$.  Thus if $n=(u,\alpha,v)\in N$, then $\alpha=\Ad
  V_{n}$ where
  \begin{equation*}
    V_{n}(h)=\alpha(\lip A(v)<h,e(v)>)f(n).
  \end{equation*}

  Hence we can define $\beta:N\to\Iso(\XX)$ by $\beta(u,\alpha,v)=(u,
  V_{(u,\alpha,v)},v)$.  To complete the proof, we just need to see
  that $\beta$ is continuous.  Suppose that
  $n_{i}=(u_{i},\alpha_{i},v_{i})\to n=(u,\alpha,v)$.  In view of
  Proposition~\ref{prop-iso-top}, it will suffice to show that if
  $h_{i}\to h$ in $\XX$ with $q(h_{i})=v_{i}$, then
  $V_{n_{i}}(h_{i})\to V_{n}(h)$.

  But we claim
  \begin{equation*}
    \lip A(v_{i})<h_{i},e(v_{i})>\to \lip A(v)<h,e(v)>
  \end{equation*}
  in $\A$.  To see this, let $f\in \H$ be such that $f(v)=h$.  Then
  $\|f(v_{i})-h_{i}\|\to 0$.  Therefore
  \begin{equation*}
    \|\lip A(v_{i})<f(v_{i}),e(v_{i})>-\lip
    A(v_{i})<h_{i},e(v_{i})>\|\to 0.
  \end{equation*}
  Since
  \begin{equation*}
    \lip A(v_{i})<f(v_{i}),e(v_{i})>\to \lip A(v) <f(v),e(v)>=\lip
    A(v)<h,e(v)>, 
  \end{equation*}
  the claim follows from \cite{wil:crossed}*{Proposition~C.20}.  Hence
  it follows from Proposition~\ref{prop-auta} that
  \begin{equation*}
    \alpha_{i}(\lip A(v_{i})<f(v_{i}),e(v_{i})>)\to \alpha(\lip
    A(v)<h,e(v)>).
  \end{equation*}
  Since $f(n_{i})\to f(n)$ in $\A$, Lemma~\ref{lem-product} implies
  that
  \begin{equation*}
    V_{n_{i}}(h_{i}) = \alpha_{i}\bigl(\lip
    A(v_{i})<h_{i},e(v_{i})>\bigr) f(n_{i}) \to \alpha\bigl(\lip
    A(v)<h,e(v)>\bigr) f(n) = V_{n}(h).
  \end{equation*}
  This completes the proof of Proposition~\ref{prop-cross-section} and
  also of Proposition~\ref{prop-new-j-map}.
\end{proof}


\begin{acknowledgements}
  The first author was partially supported by NSF grant
  DMS-1100570, and the second author by a grant from the Simons
  Foundation.
\end{acknowledgements}


\begin{bibdiv}
\begin{biblist}

\bib{fd:representations1}{book}{
      author={Fell, James M.~G.},
      author={Doran, Robert~S.},
       title={Representations of {$*$}-algebras, locally compact groups, and
  {B}anach {$*$}-algebraic bundles. {V}ol. 1},
      series={Pure and Applied Mathematics},
   publisher={Academic Press Inc.},
     address={Boston, MA},
        date={1988},
      volume={125},
        ISBN={0-12-252721-6},
        note={Basic representation theory of groups and algebras},
      review={\MR{90c:46001}},
}

\bib{gre:am78}{article}{
      author={Green, Philip},
       title={The local structure of twisted covariance algebras},
        date={1978},
     journal={Acta Math.},
      volume={140},
       pages={191\ndash 250},
}

\bib{kel:general}{book}{
      author={Kelley, John~L.},
       title={General topology},
   publisher={Van Nostrand},
     address={New York},
        date={1955},
}

\bib{kum:cjm86}{article}{
      author={Kumjian, Alexander},
       title={On {\cs}-diagonals},
        date={1986},
     journal={Canad. J. Math.},
      volume={38},
       pages={969\ndash 1008},
}

\bib{kmrw:ajm98}{article}{
      author={Kumjian, Alexander},
      author={Muhly, Paul~S.},
      author={Renault, Jean~N.},
      author={Williams, Dana~P.},
       title={The {B}rauer group of a locally compact groupoid},
        date={1998},
        ISSN={0002-9327},
     journal={Amer. J. Math.},
      volume={120},
      number={5},
       pages={901\ndash 954},
      review={\MR{2000b:46122}},
}

\bib{muh:cm01}{incollection}{
      author={Muhly, Paul~S.},
       title={Bundles over groupoids},
        date={2001},
   booktitle={Groupoids in analysis, geometry, and physics ({B}oulder, {CO},
  1999)},
      series={Contemp. Math.},
      volume={282},
   publisher={Amer. Math. Soc.},
     address={Providence, RI},
       pages={67\ndash 82},
      review={\MR{MR1855243 (2003a:46085)}},
}

\bib{muhwil:ms92}{article}{
      author={Muhly, Paul~S.},
      author={Williams, Dana~P.},
       title={Continuous trace groupoid {\cs}-algebras. {II}},
        date={1992},
     journal={Math. Scand.},
      volume={70},
       pages={127\ndash 145},
}

\bib{muhwil:plms395}{article}{
      author={Muhly, Paul~S.},
      author={Williams, Dana~P.},
       title={Groupoid cohomology and the {D}ixmier-{D}ouady class},
        date={1995},
     journal={Proc. London Math. Soc. (3)},
       pages={109\ndash 134},
}

\bib{muhwil:jams04}{article}{
      author={Muhly, Paul~S.},
      author={Williams, Dana~P.},
       title={The {D}ixmier-{D}ouady class of groupoid crossed products},
        date={2004},
        ISSN={1446-7887},
     journal={J. Aust. Math. Soc.},
      volume={76},
      number={2},
       pages={223\ndash 234},
      review={\MR{MR2041246 (2005e:46128)}},
}

\bib{muhwil:dm08}{article}{
      author={Muhly, Paul~S.},
      author={Williams, Dana~P.},
       title={Equivalence and disintegration theorems for {F}ell bundles and
  their {$C\sp *$}-algebras},
        date={2008},
        ISSN={0012-3862},
     journal={Dissertationes Math. (Rozprawy Mat.)},
      volume={456},
       pages={1\ndash 57},
      review={\MR{MR2446021}},
}

\bib{muhwil:nyjm08}{book}{
      author={Muhly, Paul~S.},
      author={Williams, Dana~P.},
       title={Renault's equivalence theorem for groupoid crossed products},
      series={NYJM Monographs},
   publisher={State University of New York University at Albany},
     address={Albany, NY},
        date={2008},
      volume={3},
        note={Available at http://nyjm.albany.edu:8000/m/2008/3.htm},
}

\bib{raewil:tams85}{article}{
      author={Raeburn, Iain},
      author={Williams, Dana~P.},
       title={Pull-backs of {$C\sp \ast$}-algebras and crossed products by
  certain diagonal actions},
        date={1985},
        ISSN={0002-9947},
     journal={Trans. Amer. Math. Soc.},
      volume={287},
      number={2},
       pages={755\ndash 777},
      review={\MR{86m:46054}},
}

\bib{rw:morita}{book}{
      author={Raeburn, Iain},
      author={Williams, Dana~P.},
       title={Morita equivalence and continuous-trace {$C^*$}-algebras},
      series={Mathematical Surveys and Monographs},
   publisher={American Mathematical Society},
     address={Providence, RI},
        date={1998},
      volume={60},
        ISBN={0-8218-0860-5},
      review={\MR{2000c:46108}},
}

\bib{simwil:nyjm13}{article}{
      author={Sims, Aidan},
      author={Williams, Dana~P.},
       title={An equivalence theorem for reduced {F}ell bundle
  {$C^*$}-algebras},
        date={2013},
     journal={New York J. Math.},
      volume={19},
       pages={159\ndash 178},
}

\bib{txlg:acens12}{article}{
      author={Tu, Jean-Louis},
      author={Xu, Ping},
      author={Laurent-Gengoux, Camille},
       title={Twisted {$K$}-theory of differentiable stacks},
        date={2004},
        ISSN={0012-9593},
     journal={Ann. Sci. \'Ecole Norm. Sup. (4)},
      volume={37},
      number={6},
       pages={841\ndash 910},
         url={http://dx.doi.org/10.1016/j.ansens.2004.10.002},
      review={\MR{2119241 (2005k:58037)}},
}

\bib{wil:crossed}{book}{
      author={Williams, Dana~P.},
       title={Crossed products of {$C{\sp \ast}$}-algebras},
      series={Mathematical Surveys and Monographs},
   publisher={American Mathematical Society},
     address={Providence, RI},
        date={2007},
      volume={134},
        ISBN={978-0-8218-4242-3; 0-8218-4242-0},
      review={\MR{MR2288954 (2007m:46003)}},
}

\end{biblist}
\end{bibdiv}

\def\noopsort#1{}\def\cprime{$'$} \def\sp{^}

\end{document}